\documentclass{article}
\usepackage{empheq}

\usepackage{spconf,amsmath,graphicx,color,lipsum,algorithm,algpseudocode,amsfonts,amscd,amsthm} 

\usepackage[usenames,dvipsnames,svgnames,table]{xcolor}

\newtheorem{definition}{Definition}
\newtheorem{theorem}{Theorem}[section]

\newcommand{\e}{\mathrm{e}}

\definecolor{orange}{rgb}{1,0.5,0}

\graphicspath{{figures/}}
\newenvironment{shrinkeq}[1]
{ \bgroup
  \addtolength\abovedisplayshortskip{#1}
  \addtolength\abovedisplayskip{#1}
  \addtolength\belowdisplayshortskip{#1}
  \addtolength\belowdisplayskip{#1}}
{\egroup}
\algdef{SE}[DOWHILE]{Do}{doWhile}{\algorithmicdo}[1]{\algorithmicwhile\ #1}

\title{Finite Synchrosqueezing Transform Based On The STFT}
\name{Mozhgan Mohammadpour$^a$, W. Bastiaan Kleijn$^b$, and Rajab Ali Kamyabi Gol$^a$}
\address{$^a$Department of Mathematics, Ferdowsi University of Mashhad, Iran.\\
$^b$School of Engineering and Computer Science, Victoria University of\\
Wellington, New Zealand}
\begin{document}
\ninept

\maketitle
\begin{abstract} 
The Finite STST Synchrosqueezing transform is a time-frequency analysis method 
that can decompose finite complex signals into time-varying oscillatory
components. This representation is sparse and invertible, allowing recovery of
the original signal. The STFT Synchrosqueezing transform on finite dimensional
signals has the advantage of an efficient matrix representation. This article
defines the finite STFT Synchrosqueezing  transform and describes some
properties of this transform. We compare the finite STFT and the finite STFT
Synchrosqueezing transform by applying these transforms to a set of signals. 
\end{abstract} 


\section{Introduction}

Time-frequency decompositions \cite{Cohen} present time-frequency information
about local frequency variations. Various time-frequency representations can be
used to analyze a signal. The Short-Time Fourier transform (STFT), e.g.,
\cite{Grochenig} is applied to extract time-frequency information of a signal
within the time-frequency plane. The continuous wavelet transform (CWT) 
\cite{Daubechies} provides a time-frequency representation of a signal with
a time and frequency localization that is appropriate for many natural
processes. The S-transform proposed by Stockwell et  al. \cite{Stockwell} is a
time-frequency analysis technique that combines elements of the CWT and the
STFT, and has been widely applied on seismic data processing and analysis of
behavior of the signal \cite{Gao}. All fore-mentioned transforms provide blurry
information about the signal \cite{Daubechies, Flandrin}. In contrast, the
Synchrosqueezing transform is a time frequency analysis that provides more
accurate information about local frequency variations.

The Synchrosqueezing transform \cite{Thakur, Thakur1, Thakur2, Wu} is a time
frequency analysis that aims to characterize time-varying oscillatory signals
efficiently. The transform is designed to analyze signals of
the form: 
\begin{equation}\label{eq0}
x\left(t\right)=\sum_{k'=1}^K A_{k'}\left(t\right)e^{2\pi i\phi_{k'}\left(t\right)},
\end{equation}
where $A_{k'}\left(t\right)$ and $\phi_{k'}\left(t\right)$ are an amplitude
function and a phase function, respectively. 
The signal $x$ consist of different layers of oscillation, each with
instantaneous time-varying information. The Synchrosqueezing transform is a
powerful tool for the 
analysis of a signal based on the notion of \textit{instantaneous
frequency}. Instantaneous frequency is a natural extension of the usual Fourier
frequency that describes how quickly a signal oscillates locally at a given point
in time or, more generally, how a quickly a number of components of the signal
oscillate locally at a given point in time. 

The Synchrosqueezing transform has a wide range of applications including 
geophysics \cite{Herrera}, global climate \cite{Thakur1},
economics\cite{Guharay} and medicine \cite{Thakur2, Wu1, Wu2}. These
applications lead in a natural manner to the consideration of the
Synchrosqueezing  transform for a finite dimensional Hilbert space. Thakur and
Wu in 
\cite{Thakur1} use a Gabor based signal analysis method that they call the STFT
Synchrosqueezing transform. In fact, they analyze continuous domain signals
$f\in \mathcal{L}^2\left(\mathbb{C}\right)$ that consist of a superposition of a
finite number of oscillatory components. Then they propose an algorithm to
recover these 
instantaneous frequencies from a finite set of uniformly or non-uniformly spaced
samples. They also prove that similar results hold for real valued signals where
$e^{2\pi i\phi_{k'}\left(t\right)}$ is replaced by $\cos\left(2\pi
  \phi_{k'}\left(t\right)\right)$ in \eqref{eq0}. This motivates us 
to focus on the 
oscillatory decomposion of finite-dimensional signals and define the STFT
Synchrosqueezing transform for them. 

The STFT Synchrosqueezing transform on finite groups can be studied in the realm
of numerical linear algebra. In the other words, it is defined based on the STFT
on finite groups. However, the STFT Synchrosqueezing transform is based on the
continuous STFT transform. The finite STFT has a matrix representation which is
more convenient to work rather than the continuous STFT. It provides
matrix factorization and for this reason finite short time Fourier analysis is more
popular in engineering rather than the STFT.  Moreover, the 
STFT Synchrosqueezing transform on a finite cyclic group can be studied
numerically in order to better understand properties of STFT Synchrosqueezing
transform on the real line. Similarly to the STFT transform, the relationships
between the STFT Synchrosqueezing transform on the real line, on the integers,
and on cyclic groups can be studied based on a sampling and periodization
argument \cite{Thakur1, Wu}. On the other hand, the delta function, which is
used in the STFT Synchrosqueezing transform, can not be defined on 
$\mathcal{L}^2\left(\mathbb{C}\right)$ and an approximation delta is used
instead of delta function. However, it is defined on $\mathbb{C}^N$. 

In this paper, we consider that the continuous signal is sampled by $N$
regularly spaced samples. In other words, we assume $\phi_{k'}$ and $A_{k'}$ are
naturally continuous functions but we are given only $\phi_{k'}\left(n\right)$
and $A_{k'}\left(n\right)$ for $n=0,\cdots,N-1$. As a result, we work with a
periodic discrete function $x:\mathbb{Z}\rightarrow \mathbb{C}$ where
$x\left(k+N\right)=x\left(k\right)$ for $k\in \mathbb{Z}$. It is of the form:
\begin{equation}\label{eq1}
x\left(n\right)= \sum_{k'=1}^K A_{k'}\left(n\right) \e^{2\pi i \frac{\phi_{k'}\left(n\right)}{N}}, \quad n=0,1,\cdots, N-1.
\end{equation}
This paper structured as follows. We start with a brief introduction of the
STFT transform for finite dimensional signals and some of its properties in
section II. In section III, we discuss the instantaneous frequency information
and the finite Synchrosqueezing transform for finite, purely harmonic
signals. Then, we demonstrate how to use the definition of
instantaneous frequency information of pure harmonic signals to approximate the
instantaneous frequency of an oscillatory signal. Then we define the finite
Synchrosqueezing transform and show how we can extract the oscillatory
components. In section IV we show that the results given in section
III also are true for real signals where the exponentials $\e^{2\pi i
\frac{\phi_{k'}\left(n\right)}{N}}$ are replaced by $\cos\left(2\pi 
\frac{\phi_{k'}\left(n\right)}{N}\right)$ in \eqref{eq1}. In Section V, we
show that the finite Synchrosqueezing transform has the stability
property. Section VI is devoted to some comparative numerical results. 

\section{\bf{Preliminaries and Notations}} 

To define the STFT-Synchrosqueezing transform for finite dimensional signals, we 
first recall the short time Fourier transform for finite dimensional
signals. We index the components of a vector $x\in \mathbb{C}^N$ by 
$\{0, 1, \cdots, N-1\}$, i.e., the cyclic group
$\mathbb{Z}_{N}$. We will write $x\left( k \right)$ 
instead of $\mathbf{x}_k$ to avoid algebraic operations on
indices. 

The discrete Fourier transform is basic in short time Fourier analysis and is defined as
\[ 
\mathcal{F} x \left( m \right) = \hat{x}\left( m \right) = \sum_{n=0}^{N-1} x\left( n\right) e^{-2\pi i m\frac{n}{N}}. 
\] 

The most important properties of the Fourier transform are the Fourier
inversion formula and the Plancherel formula
\cite{Bastiaans,Qiu1,Qiu2,Pfander}. 
The inversion formula shows that
any $x$ can be written as a linear combination of harmonics.
This means the normalized harmonics $\{\frac{1}{\sqrt{N}} e^{2\pi i m
\frac{\left(. \right)}{N}}\}_{m=0}^{N-1}$ form an orthonormal basis
of $\mathbb{C}^N$ and hence we have 
\[ 
x= \frac{1}{N} \sum_{m=0}^{N-1} \hat{x} \left( m \right) e^{2\pi i m \frac{\left(. \right)}{N}} \quad x \in \mathbb{C}^N. 
\] 
Moreover, the Plancherel formula states
\[ 
\langle x, y \rangle = \frac{1}{N} \langle \hat{x}, \hat{y} \rangle \quad x, y \in \mathbb{C}^N,
\] 
which results in 
\[ 
\sum_{n=0}^{N-1} \vert x \left( n \right) \vert ^2 = \frac{1}{N} \sum_{m=0}^{N-1} \vert \hat{x} \left( m \right) \vert ^2,
\] 
where $\vert x\left(n \right) \vert ^2$ quantifies the energy
of the signal $x$ at time $n$, and where the right-hand side 
indicates that the harmonic $e^{2\pi
i m\frac{\left(.\right)}{N}}$ contributes energy $\frac{1}{N}\vert
\hat{x}\left(m \right) \vert^2$ to $x$.

Short time Fourier analysis is the interplay of the Fourier transform, translation operators, and modulation operators. The cyclic translation operator $T_k: \mathbb{C}^N \rightarrow \mathbb{C}^N$ is given by

\begin{align*}
T_k&\mathbf{x}= T_k \left( x\left(0 \right), \cdots, x \left( N-1 \right) \right)^t\\
&=\left(x\left( N-k\right) , x\left( N-k+1\right),x \left( 0 \right) , \cdots, x \left( N-k-1 \right) \right)^t.
\end{align*}
It is seen that the operator $T_k$ alters the position of the entries of
$x$ and that, equivalently, $n-k$ is achieved modulo $N$. 

Similarly, the modulation operator $M_l:\mathbf{C}^N\rightarrow \mathbb{C}^N$ is
given by 
\begin{align*}
M_l&x= \left( e^{-2\pi i l\frac{0}{N}}x\left( 0 \right), \right.\\
&\left. e^{-2\pi i l\frac{1}{N}} x\left(1\right), \cdots , e^{-2\pi i l\frac{N-1}{N}} x\left( N-1 \right) \right)^t.
\end{align*}
The modulation operators are implemented as the pointwise product of the
vector with harmonics $e^{-2\pi i l\frac{.}{N}}$. 

The translation and modulation operators are time-shift and frequency-shift
operators, respectively. The time-frequency shift operator $\pi \left( k,l \right)$
is the combination of translation operators and modulation operators:
\[ 
\pi \left(k,l \right): \mathbb{C}^N\rightarrow \mathbb{C}^N \quad \pi \left( k,l\right) x= M_lT_k x.
\] 
Hence, the STFT 
$V_{g}: \mathbb{C}^N\rightarrow \mathbb{C}^{N\times N}$ with respect to the window
$g \in \mathbb{C}^N$ for every $x \in \mathbb{C}^N$ can be written as
\[ 
V_{\phi}x \left( k,l \right) = \langle x, \pi \left( k,l \right) \phi \rangle= \sum_{n=0}^{N-1} x\left(n \right) \overline{g\left(n-k\right)} e^{-2\pi i l\frac{n}{N}}, 
\] 
where $\overline{g}$ is the conjugate of $g$. The STFT 
uses a window function $g$, supported at a neighborhood of zero that is translated by
$k$. Hence, the pointwise product with $\mathbf{x}$ selects a portion
of $\mathbf{x}$ centered in time at $k$, and this portion is analyzed using a
Fourier transform.

\section{\bf{Finite Synchrosqueezing Transform of Harmonic Signal}}
\label{s:STFTS} 
In this section, we show how we can estimate the instantaneous frequency
information based on short time Fourier transform.  We first motivate the idea of  
finding instantaneous freuency information by considering the instantaneous
frequency of a purely harmonic signal,  
\begin{equation}
x\left(n\right)=Ae^{\frac{2\pi i\omega n}{N}}.
\label{q:a_simple_x}
\end{equation}
That is, we first define the instantaneous frequency information for a purely harmonic signals. Then we extend this
definition to a combination of elementary oscillations. To introduce
an instantaneous frequency information formula, we define the modified STFT:
\begin{align*}
V_{g}&x \left( n,l \right) = \langle x, T_nM_l g \rangle= \sum_{k=0}^{N-1} x\left(k \right) \overline{g\left(k-n\right)} e^{-2\pi i l\frac{k-n}{N}}\\
&=e^{2\pi i l\frac{n}{N}}\sum_{k=0}^{N-1} x\left(k \right) \overline{g\left(k-n\right)} e^{-2\pi i l\frac{k}{N}}.
\end{align*}

Consider a smooth window function $g$ is given. By Plancherel's theorem, we can write $V_{g}x$,
the short time Fourier transform of $x$ of \eqref{q:a_simple_x} with respect to $g$ as 
\begin{align*}
V_{g}x\left(n,l\right) &=\langle x, T_nM_lg\rangle =\frac{1}{N} \langle \hat{x}, M_{-n}T_l\hat{g}\rangle\\
&=\frac{1}{N} \sum_{\xi=0}^{N-1} \hat{x}\left( \xi\right) \overline{T_l\hat{g}\left(\xi \right)} e^{\frac{2\pi i n \xi}{N}}\\
&= \frac{A}{N}\sum_{\xi=0}^{N-1}\delta \left( \xi - \omega \right)\overline{\hat{g}\left( \xi - l \right)}e^{\frac{2\pi i n \xi}{N}}\\
&= \frac{A}{N} \hat{g}\left(\omega-l\right)e^{\frac{2\pi i n \omega}{N}} .
\end{align*}
Motivated by the above analysis, we now define for a signal $x$ of the form 
\eqref{q:a_simple_x}, for any $\left(n,l\right)$ where
$V_{g}x\left(n,l\right)\neq0$, the instantaneous frequency information
$\omega_x\left(n,l\right)$ as 
\begin{equation}
\omega_x\left(n,l\right)=\mathrm{round}\left(\frac{1}{\frac{2\pi i}{N}}\mathrm{ln}\left(\frac{V_{g}x\left(n+1,l\right)}{V_{g}x\left(n,l\right)}\right)\right).
\label{q:instfreq}
\end{equation}
It should be noted that the instantaneous frequency information must be an
integer for the defininition of the STFT Synchrosqueezing transform. However,
$\omega$ is not necessary integer. This is the reason why we need the rounding
function to find the nearest integer number to the instantaneous frequency. The
rounding function used in \eqref{q:instfreq} is given by 
\[
\mathrm{round}\left(x\right)=\left\{ 
\begin{array}{ll}
\left[x\right] & x\leq \left[x\right]+0.5\\
\left[x+1\right]& x> \left[x\right]+0.5
\end{array}
\right.
\]
where $\left[x\right]$ is the integer part of $x$. 

Note that the instantantaneous frequency information, in general, does not need
to be equal 
to the instantaneous frequency of a given signal. In the context of the STFT
Synchrosqueezing transform this implies that $\omega_x(n,l) $, in general,
differs from the frequency $l$. In fact, as will be shown below, the transform
``squeezes''  energy to the frequency $l$ that equals the instaneous
frequency information $\omega_x(n,l) $.

We now are ready to define the STFT Synchrosqueezing based on the phase function for $x$ of the form \eqref{q:a_simple_x}:
\[
S_{\phi}x\left( n, \xi\right) =\sum_{l=0}^{N-1}V_{\phi}x\left(n,l\right)\delta \left(\xi - \omega_x\left(n,l\right)\right).
\]
Thus, the Synchrosqueezing transform squeezes the energy in the frequency
direction to the instantaneous frequency information of the purely harmonic
signal. 

Next, based on the definition of instantaneous frequency information of
purely harmonic signals and the STFT Synchrosqueezing transform, we 
generalize these notions for oscillatory signals with multiple components. To this 
purpose, we define a function class of functions with well-separated oscillatory
components: 
\begin{definition}\label{def1}
The space $\mathcal{A}_d\subset \mathbb{C}^N$ of superpositions consist of functions $x$ having the form 
\[
x\left(n\right)= \sum_{k'=1}^K x_{k'}\left(n\right),
\]
for some $K> 0$ and $x_{k'}\left(n\right)=A_{k'}\left(n\right)e^{\frac{2\pi \phi_{k'}\left(n\right)}{N}}$ and $\phi_{k'}$ satisfy
\[
\phi'_{k'}\left(n\right)-\phi'_{k'+1}\left(n\right)>d\quad \textrm{for every\,\,} n\in \mathbb{Z}_N,
\]
where $d\leq N-1$
\end{definition}
Functions in $\mathcal{A}_d$ are composed of a summation of oscillatory components and
the instantaneous frequencies of any two consecutive components are separated by
at least $d$. For a given window function $g\in \mathbb{C}^N$, we can apply the
modified short time Fourier transform to any $x\in\mathcal{A}_d $. Similarly, we
can define the instantaneous frequency information as: 
\begin{definition}[Instantaneous frequency information function] 
Let $x\in \mathcal{A}_d$. Select a window function $\phi\in \mathbb{C}^N$ such
that $\vert\mathrm{supp}\left(\hat{g}\right)\vert<\frac{d}{2}$. The
instantaneous frequency information $\omega_x\left(k,l\right)$ is defined by 
\begin{equation}\label{eq2}
\omega_x\left(n,l\right)=\left\{
\begin{array}{ll}
\mathrm{round}\left(\mathrm{real}\left(\frac{1}{\frac{2\pi i }{N}}\mathrm{ln}\left(\frac{V_{g}x\left(n+1,l\right)}{V_{g}x\left(n,l\right)}\right)\right)\right) & V_{g}x\left(n,l\right)\neq0\\
0& V_{g}x\left(n,l\right)=0
\end{array}
\right.
\end{equation}
\end{definition}

We define the finite Synchrosqueezing transform based on the instantaneous frequency information as follows:
\begin{definition}
Let $x \in \mathcal{A}_d$. Select a window function $g\in \mathbb{C}^N$ such that $\vert\mathrm{supp}\left(\hat{g}\right)\vert<\frac{d}{2}$. The STFT Synchrosqueezing transform is defined by
\[
S_{g}x\left( n, \xi\right) =\sum_{l=0}^{N-1}V_{g}x\left(n,l\right)\delta\left(\xi-\omega_x\left(n,l\right)\right).
\]
\end{definition}
The following theorem states that the error between instantaneous
frequency $\phi'_{k_0}$ and instantaneous frequency information 
$\omega_x\left(n,l\right)$ is less than a given $\epsilon$. Before we
provide the Theorem, we define some notations:
\[
I_s^1=\sum_{k=0}^{N-1} \vert g\left(k-n \right) \vert \vert k-n\vert^s
\]
and
\[
I_s^2=\sum_{k=0}^{N-1} \vert g\left(k-n-1\right) \vert \vert k-n\vert^s.
\]
Moreover, we denote $Z_{k'}=\{\left(n,l\right): \vert
l-\phi_{k'}\left(n\right)\vert<\frac{d}{2}\}$ for each $k'\in\{1,
\cdots,K\}$.

We can now state the main theorem with related to the finite STFT Synchrosqueezing transform:
\begin{theorem}\label{th2}
Consider a signal $x\in \mathcal{A}_d$ is given. Select a window function $g\in \mathbb{C}^N$ such that $\vert\mathrm{supp}\left(\hat{g}\right)\vert<\frac{d}{2}$. Moreover consider $\Vert\phi''_{k'}\Vert_{\infty}\leq \epsilon\Vert\phi'_{k'}\Vert_{\infty}$ and $\Vert A'_{k'}\Vert_{\infty}\leq \epsilon\Vert\phi'_{k'}\Vert_{\infty}$. Also, let $\delta=\min\{V_gx\left(n,l\right); 0\leq n,l \leq N-1, V_gx\left(n,l\right)\neq 0\}$ and $M=\max\{1,\frac{V_gx\left(n,l\right)}{V_gx\left(n+1,l\right)}\}$. Then, we have the following
\begin{equation}\label{ff2}
\begin{split}
&\vert \omega_x\left(n,l\right)-\phi'_{k_0}\left(n\right)\vert \\
&\leq \frac{NM\epsilon}{2\pi \delta}\left[ \sum_{k'=1}^K \Vert \phi'_{k'}\Vert_{\infty} \left( I_1^1+\pi \vert A_{k'}\left(n\right)\vert I_2^1+ I_1^2+\pi \vert A_{k'}\left(n\right)\vert I_2^2\right) \right]\\
&\quad +0.5=\tilde{\epsilon}
\end{split}
\end{equation}
where $\left(n,l\right)\in Z_{k_0}$.

Furthermore, if $\left(n,l\right)\notin Z_{k'}$ for every $k'=1,\cdots,K$, then
\begin{equation}
\label{eqt1}
\vert V_gx\left(n,l\right)\vert\leq \sum_{k'=0}^{K} \epsilon \Vert\phi '_{k'}\Vert_{\infty} \left( I_1^1+\pi \vert A_{k'}\left(n\right) \vert I_2^1\right).
\end{equation}
and
\begin{equation}\label{eqt2}
\vert V_gx\left(n+1,l\right)\vert\leq \sum_{k'=0}^{K} \epsilon \Vert\phi '_{k'}\Vert_{\infty} \left( I_1^2+\pi \vert A_{k'}\left(n\right) \vert I_2^2\right).
\end{equation}
\end{theorem}

Theorem \ref{th2} basically tells us how to recover instantaneous frequencies
for multi-component oscillatory signals. It shows that the error of
instantaneous frequencies obtained from the formula \eqref{eq2} is less than 
$\tilde{\epsilon}$. Next we prove the theorem.
\begin{proof}
To prove the inequality \eqref{ff2} we have
\small{
\begin{align}
&\vert \omega_x\left(n,l\right)-{\phi'}_{k_0}\left(n\right)\vert\nonumber \\
&=\left\vert \mathrm{round}\left(\mathrm{real}\left(\frac{N}{2\pi i}\mathrm{ln}\left(\frac{V_{g}x\left(n+1,l\right)}{V_{g}x\left(n,l\right)}\right)\right)\right) -\phi '_{k_0}\left(n\right)\right\vert\nonumber\\
&\leq\left\vert\frac{N}{2\pi i}\mathrm{ln}\left(\frac{V_{g}x\left(n+1,l\right)}{V_{g}x\left(n,l\right)}\right) -\phi '_{k_0}\left(n\right)\right\vert+0.5\nonumber\\
&=\frac{N}{2\pi}\left\vert \mathrm{ln}\left(\frac{V_{g}x\left(n+1,l\right)}{V_{g}x\left(n,l\right)}\right)-\mathrm{ln}e^{\frac{2\pi i\phi '_{k_0}\left(n\right)}{N}}\right\vert+0.5\nonumber\\
&\leq \frac{N}{2\pi }M\left\vert \frac{V_{g}x\left(n+1,l\right)}{V_{g}x\left(n,l\right)}-e^{\frac{2\pi i\phi '_{k_0}\left(n\right)}{N}} \right\vert+0.5 \label{eq3}\\
&\leq \frac{N}{2\pi}M\nonumber\\
&\left[\left\vert \frac{V_gx\left(n+1,l\right)-e^{\frac{2\pi i\phi'_{k_0}\left(n\right)}{N}}A_{k_0}\left(n\right)e^{\frac{2\pi i \phi_{k_0}\left(n\right)}{N}}\hat{g}\left(l-\phi'_{k_0}\left(n\right)\right)}{V_gx\left(n,l\right)} \right\vert \right.\nonumber\\
&\left. +\left\vert e^{\frac{2\pi i\phi'_{k_0}\left(n\right)}{N}}\frac{A_{k_0}\left(n\right)e^{\frac{2\pi i\phi_{k_0}\left(n\right)}{N}}\hat{g}\left(l-\phi'_{k_0}\left(n\right)\right)-V_gx\left(n,l\right)}{V_gx\left(n,l\right)}\right\vert \right]\nonumber\\
&+0.5\label{eq33}
\end{align}}

We derived \eqref{eq3} by the mean value theorm. 

Note that 
$A_{k_0}\left(n\right)e^{\frac{2\pi i \phi_{k_0}\left(n\right)}{N}}\hat{g}\left(l-\phi'_{k_0}\left(n\right)\right)=0$ for $\left(n,l\right)\notin Z_{k'}$ for any $k'=1,\cdots, K$, since $\vert\mathrm{supp}\left(\hat{g}\right)\vert<\frac{d}{2}$. Hence we can rewrite the second part of \eqref{eq33} as:
\begin{shrinkeq}{-0.4ex}
\begin{align}
\vert & A_{k_0}\left(n\right)e^{\frac{2\pi i
\phi_{k_0}\left(n\right)}{N}}\hat{g}\left(l-\phi'_{k_0}\left(n\right)\right)-V_gx\left(n,l\right)\vert\nonumber\\
&=\vert\sum_{k'=1}^K \left(A_{k'}\left(n\right)e^{\frac{2\pi i \phi_{k'}\left(n\right)}{N}}\sum_{k=0}^{N-1}g\left(k-n\right)e^{\frac{-2\pi i}{N}\left(l-\phi'_{k'}\left(n\right)\right)\left(k-n\right)} \notag \right.\nonumber\\
&\left. -\sum_{k=0}^{N-1} A_{k'} \left(k\right)e^{\frac{2\pi i\phi_{k'}\left(k\right)}{N}}g\left(k-n\right)e^{\frac{-2\pi il\left(k-n\right)l}{N}}\right)\vert\nonumber\\
&\leq \sum_{k'=1}^K\left(\sum_{k=0}^{N-1}\vert A_{k'}\left(k\right)-A_{k'}\left(n\right)\vert\vert g\left(k-n\right)\vert \notag \right.\nonumber\\
& \left. +\sum_{k=0}^{N-1}\vert A_{k'}\left(n\right)g\left(k-n\right)\vert\vert e^{\frac{2\pi i \phi_{k'}\left(k\right)}{N}}-e^{\frac{2\pi i}{N}\left[\phi_{k'}\left(n\right)
+\phi'_{k'}\left(n\right)\left(k-n\right)\right]}\vert\right)\nonumber\\
& \leq \sum_{k'=1}^K \left(\Vert A_{k'}'\Vert_{\infty} \sum_{k=0}^{N-1} \vert g\left(k-n \right) \vert \vert k-n\vert \notag \right.\nonumber\\
&\left.+ \vert A_{k'}\left(n\right) \vert \sum_{k=0}^{N-1} \vert g\left(k-n\right) \vert \pi \Vert \phi''_{k'} \Vert_{\infty} \vert k-n \vert ^2\right)\label{eq4}\\
&\leq\sum_{k'=0}^{K}\epsilon\Vert\phi '_{k'}\Vert_{\infty}\left( I_1^1+\pi \vert A_{k'}\left(n\right)\vert I_2^1\right).\label{eq4p}
\end{align}
We can derive \eqref{eq4} using a Taylor expansion:
\begin{align*}
&\vert 1-e^{\frac{2\pi i}{N}\left[-\phi_{k'}\left(k\right)+\phi_{k'}\left(n\right)+\phi'_{k'}
\left(n\right)\left(k-n\right)\right]}\vert\\
& \leq \pi \Vert {\phi''}_{k'}\Vert_{\infty}\vert k-n\vert^2.
\end{align*}
The first part of inequality \eqref{eq33} can be written as:
\begin{align}
\vert & V_gx\left(n+1,l\right)- e^{\frac{2\pi
i}{N}\left(\phi'_{k_0}\left(n\right)\right)}
A_{k_0}\left(n\right)e^{\frac{2\pi i
\phi_{k_0}\left(n\right)}{N}}\hat{g}\left(l-\phi'_{k_0}\left(n\right)\right)\vert.\nonumber\\
&=\vert\sum_{k'=1}^K \left( \sum_{k=0}^{N-1} A_{k'}\left(k\right) e^{\frac{2\pi
i\phi_{k'}\left(k\right)}{N}}g\left(k-n-1\right)e^{\frac{-2\pi
il\left(k-n-1\right)}{N}} \notag\right.\nonumber\\
&\left.-\sum_{k=0}^{N-1}A_{k'}\left(n\right)g\left(k-n-1\right)e^{\frac{-2\pi i}{N}\left(l-\phi'_{k'}\left(n\right)\right)\left(k-n-1\right)} \notag \right.\nonumber\\
&\left. e^{\frac{2\pi i{\phi'}_{k'}\left(n\right)}{N}}e^{\frac{2\pi i\phi_{k'}\left(n\right)}{N}} \right)\vert\nonumber\\
&\leq\sum_{k'=1}^K \sum_{k=0}^{N-1}\vert g\left(k-n-1\right)\vert\nonumber \\
&\vert A_{k'}\left(k\right) e^{\frac{2\pi i \phi_{k'}\left(k\right)}{N}}-A_{k'}\left( n\right) e^{\frac{2\pi i}{N}\left[\phi\left(n\right)+\phi'_{k'}\left(n\right)\left(k-n\right)\right]}\vert\nonumber\\
&\leq\sum_{k'=1}^K\sum_{k=0}^{N-1}\vert g\left(k-n-1\right)\vert \left[ \vert A_{k'}\left(k\right)\notag \right.\nonumber\\
&\left. -A_{k'}\left(n\right)\vert+\vert A_{k'}\left(n\right)\vert \vert e^{\frac{2\pi i \phi_{k'}\left(k\right)}{N}}-e^{\frac{2\pi i}{N}\left[\phi_{k'}\left(n\right)+\phi'_{k'}\left(n\right)\left(k-n\right)\right]}\vert \right]\nonumber\\
&\leq \sum_{k'=1}^K \left( \Vert A_{k'}'\Vert_{\infty} \sum_{k=0}^{N-1} \vert g\left(k-n-1 \right) \vert \vert k-n\vert\notag \right.\nonumber\\
&\left.+ \vert A_{k'}\left(n\right) \vert \sum_{k=0}^{N-1} \vert g\left(k-n-1\right) \vert \pi \Vert \phi''_{k'} \Vert_{\infty} \vert k-n \vert ^2 \right)\nonumber\\
& \leq \sum_{k'=0}^{K} \epsilon \Vert\phi '_{k'}\Vert_{\infty} \left( I_1^2+\pi \vert A_{k'}\left(n\right) \vert I_2^2\right).\label{eq5}
\end{align}
\end{shrinkeq}
By considering the inequalities \eqref{eq33}, \eqref{eq4p} and \eqref{eq5} we have
\begin{shrinkeq}{-0.4ex}
\begin{align*}
&\vert\omega_x\left(n,l\right)-\phi'_{k_0}\left(n\right)\vert \leq \\
&\frac{NM\epsilon}{2\pi \delta}\left[ \sum_{k'=1}^K \Vert \phi'_{k'}\Vert_{\infty} \left( I_1^1+\pi \vert A_{k'}\left(n\right)\vert I_2^1+ I_1^2+\pi \vert A_{k'}\left(n\right)\vert I_2^2\right) \right]+0.5=\tilde{\epsilon}.
\end{align*}
\end{shrinkeq}
Now by the fact that $A_{k_0}\left(n\right)e^{\frac{2\pi i\phi_{k_0}\left(n\right)}{N}} \hat{g} \left(l-{\phi'}_{k_0}\left(n\right)\right)=0$ for $\left(n,l\right)\notin Z_{k'}$ and by \eqref{eq4p} and \eqref{eq5} we can derive the inequalities \eqref{eqt1} and \eqref{eqt2}.
\end{proof}
For an appropriate signal $x\in \mathbb{C}^N$, the energy in the
Synchrosqueezing transform $S_gx\left(n,\xi\right)$ is concentrated around the
instantaneous frequency curves 
$\{\phi'_{k'}\left(n\right)\}$.  Once $S_gx$ is computed, we can recover each of the components by completing the invertion of the STFT and take the summation over small bands around each instantaneous frequency curve. 
\begin{theorem}\label{th3}
Consider a signal $x\in \mathcal{A}_d$ and a window function $g\in \mathbb{C}^N$ are given such that $\vert\mathrm{supp}\left(\hat{g}\right)\vert<\frac{d}{2}$. Moreover, assume $\Vert\phi''_{k'}\Vert_{\infty}\leq \epsilon\Vert\phi'_{k'}\Vert_{\infty}$ and $\Vert A'_{k'}\Vert_{\infty}\leq \epsilon\Vert\phi'_{k'}\Vert_{\infty}$. Let $\delta=\min\{V_gx\left(n,l\right); 0\leq n,l \leq N-1, V_gx\left(n,l\right)\neq 0\}$ and $M=\max\{1,\frac{V_gx\left(n,l\right)}{V_gx\left(n+1,l\right)}\}$. Then for each $n=0,1,\cdots,N-1$ and $k'\in\{1,2,\cdots,K\}$ we have
\[
A_{k'}\left(n\right) e^{\frac{2\pi i}{N} \phi_{k'}\left(n\right)}=\sum_{\vert\xi-\phi'_{k'}\left(n\right)\vert \leq \epsilon}S_gx\left(n,\xi\right)
\]
\end{theorem}
\begin{proof}
we have
\begin{equation}\label{eqp1}
\begin{split}
&\sum_{\vert\xi-\phi'_{k'}\left(n\right)\vert \leq \epsilon}S_gx\left(n,\xi\right)\\
&= \sum_{\vert\xi-\phi'_{k'}\left(n\right)\vert \leq \tilde{\epsilon}}\sum_{l=0}^{N-1}V_gx\left(n,l\right) \delta \left(\xi-\omega_x\left(n,l\right)\right)\\
&=\sum_{l=0}^{N-1}V_gx\left(n,l\right) \sum_{\vert\xi-\phi'_{k'}\left(n\right)\vert \leq \tilde{\epsilon}} \delta \left(\xi-\omega_x\left(n,l\right)\right)\\
&=\sum_{\{l;\vert\omega_x\left(n,l\right)-\phi'_{k'}\left(n\right)\vert \leq \tilde{\epsilon}\}}V_gx\left(n,l\right)
\end{split}
\end{equation}
Now we obtain 
\begin{align*}
&\frac{1}{\overline{g}\left( 0\right)}\sum_{\vert\xi-\phi'_{k'}\left(n\right)\vert \leq \tilde{\epsilon}}S_gx\left(n,\xi\right)\\
&=\frac{1}{\overline{g}\left( 0\right)}\sum_{\{l;\vert l-\phi'_{k'}\left(n\right)\vert \leq \frac{d}{4}\}}V_gx\left(n,l\right)\\
&=\frac{1}{\overline{g}\left( 0\right)}A_{k'}\left(n\right) e^{\frac{2\pi i}{N} \phi_{k'}\left(n\right)}\sum_{\{l;\vert l-\phi'_{k'}\left(n\right)\vert \leq\frac{d}{4}\}}\hat{g}\left(l-\phi'_{k'}\left(n\right)\right)\\
&=A_{k'}\left(n\right) e^{\frac{2\pi i}{N} \phi_{k'}\left(n\right)}
\end{align*}
\end{proof}
\section{\bf{Finite Synchrosqueezing Transform of Real Signals}}
\label{s:STFTSreal}
In this section we show that results similar to those in section  \ref{s:STFTS} hold
true for the real valued signal $x$, where the exponentials $e^{2\pi i
  \frac{\phi_{k'}\left(n\right)}{N}}$ are replaced by $\cos\left(2\pi
  \frac{\phi_{k'}\left(n\right)}{N}\right)$ in \eqref{eq1}.  
Hence we define the function space $\mathcal{B}_d$ as all $x\in \mathcal{A}_d$
such that $e^{2\pi i \frac{\phi_{k'}\left(n\right)}{N}}$ are replaced by
$\cos\left(2\pi \frac{\phi_{k'}\left(n\right)}{N}\right)$ in the definition
\ref{def1}.  

In section \ref{s:STFTSreal} we proved that the error of instantaneous
frequencies obtained from\eqref{eq2} for a signal $x\in
\mathcal{A}_d$ is less than a given $\tilde{\epsilon}$. Now we show a similar
result for $x\in \mathcal{B}_d$. 

\begin{theorem}\label{th4}
Consider a signal $x\in \mathcal{B}_d$ is given. Select a window function $g\in \mathbb{C}^N$ such that $\vert\mathrm{supp}\left(\hat{g}\right)\vert<\frac{d}{2}$. Moreover consider $\Vert\phi''_{k'}\Vert_{\infty}\leq \epsilon\Vert\phi'_{k'}\Vert_{\infty}$ and $\Vert A'_{k'}\Vert_{\infty}\leq \epsilon\Vert\phi'_{k'}\Vert_{\infty}$. Also, let $\delta=\min\{V_gx\left(n,l\right); 0\leq n,l \leq N-1, V_gx\left(n,l\right)\neq 0\}$ and $M=\max\{1,\frac{V_gx\left(n,l\right)}{V_gx\left(n+1,l\right)}\}$. Then, for any $\left(n,l\right)\in Z_{k_0}$ we have the following:
\begin{equation}\label{f2}
\begin{aligned}
\vert &\omega_x\left(n,l\right)-\phi'_{k_0}\left(n\right)\vert\\
&\leq\frac{NM\epsilon}{2\pi \delta}\left[ \sum_{k'=1}^K \Vert {\phi'}_{k'}\Vert_{\infty} \left( I_1^1+\pi \vert A_{k'}\left(n\right)\vert I_2^1+I_3^1\notag \right.\right.\\
&\left.\left. + I_1^2+\pi \vert A_{k'}\left(n\right)\vert I_2^2+I_3^2\right) \right]+0.5=\tilde{\epsilon},
\end{aligned}
\end{equation}

where $I_i^j$ for $i,j=1,2$ is denoted the same as in the section $3$ and $I_3^j$ for $j=1,2$ is denoted as
\[
I_3^j=\sum_{k=0}^{N-1} \vert g\left(k-n-j+1\right) \vert \vert k+1\vert.
\]

Furthermore, if $\left(n,l\right)\notin Z_{k'}$ for every $k'=1,\cdots,K$, then
\begin{equation}\label{eqt3}
\vert V_gx\left(n,l\right)\vert\leq \sum_{k'=0}^{K} \epsilon \Vert\phi '_{k'}\Vert_{\infty} \left( I_1^1+\pi \vert A_{k'}\left(n\right) \vert I_2^1+I_3^1\right).
\end{equation}
and
\begin{equation}\label{eqt4}
\vert V_gx\left(n+1,l\right)\vert\leq \sum_{k'=0}^{K} \epsilon \Vert\phi '_{k'}\Vert_{\infty} \left( I_1^2+\pi \vert A_{k'}\left(n\right) \vert I_2^2+I_3^2\right).
\end{equation}
\end{theorem}
\begin{proof}
Similar to the proof of Theorem \ref{th2} we have 
\begin{shrinkeq}{-0.4 ex}
\begin{align}
&\vert \omega_x\left(n,l\right)-{\phi'}_{k_0}\left(n\right)\vert\nonumber\\ 
&=\left\vert \mathrm{round}\left(\mathrm{real}\left(\frac{N}{2\pi i}\mathrm{ln}\left(\frac{V_{g}x\left(n+1,l\right)}{V_{g}x\left(n,l\right)}\right)\right)\right) -\phi '_{k_0}\left(n\right)\right\vert \leq \frac{N}{2\pi}M\nonumber\\
&\left[\left\vert \frac{V_gx\left(n+1,l\right)-e^{\frac{2\pi i\phi'_{k_0}\left(n\right)}{N}}A_{k_0}\left(n\right)e^{\frac{2\pi i \phi_{k_0}\left(n\right)}{N}}\hat{g}\left(l-\phi'_{k_0}\left(n\right)\right)}{V_gx\left(n,l\right)} \right\vert \right.\nonumber\\ 
&\left. +\left\vert e^{\frac{2\pi i\phi'_{k_0}\left(n\right)}{N}}\frac{A_{k_0}\left(n\right)e^{\frac{2\pi i\phi_{k_0}\left(n\right)}{N}}\hat{g}\left(l-\phi'_{k_0}\left(n\right)\right)-V_gx\left(n,l\right)}{V_gx\left(n,l\right)}\right\vert \right]\nonumber\\
&+0.5 \label{eq66}
\end{align}
\end{shrinkeq}
Now considering $A_{k_0}\left(n\right)e^{\frac{2\pi i \phi_{k_0}\left(n\right)}{N}}\hat{g}\left(l-\phi'_{k_0}\left(n\right)\right)=0$ for $\left(n,l\right)\notin Z_{k'}$ for any $k'=1,\cdots, K$, since $\vert\mathrm{supp}\left(\hat{g}\right)\vert<\frac{d}{2}$, we can rewrite the second part of the last inequality of \eqref{eq66} as:

\begin{align}
\vert & A_{k_0}\left(n\right)e^{\frac{2\pi i
\phi_{k_0}\left(n\right)}{N}}\hat{g}\left(l-\phi'_{k_0}\left(n\right)\right)-V_gx\left(n,l\right)\vert\nonumber\\
&=\vert\sum_{k'=1}^K \left(A_{k'}\left(n\right)e^{\frac{2\pi i \phi_{k'}\left(n\right)}{N}}\sum_{k=0}^{N-1}g\left(k-n\right)e^{\frac{-2\pi i}{N}\left(l-\phi'_{k'}\left(n\right)\right)\left(k-n\right)} \notag \right.\nonumber \\
&\left. -\sum_{k=0}^{N-1} A_{k'} \left(k\right)\cos\left(\frac{2\pi \phi_{k'}\left(k\right)}{N}\right)\left(k-n\right)e{\frac{-2i\pi l\left(k-n\right)l}{N}}\right)\vert\nonumber \\
&\leq\sum_{k'=1}^K \left(\vert A_{k'}\left(n\right)e^{\frac{2\pi i \phi_{k'}\left(n\right)}{N}}\sum_{k=0}^{N-1}g\left(k-n\right)e^{\frac{-2\pi i}{N}\left(l-\phi'_{k'}\left(n\right)\right)\left(k-n\right)} \notag \right.\nonumber  \\
&\left. -\sum_{k=0}^{N-1} A_{k'} \left(k\right)e^{\frac{2\pi i\phi_{k'}\left(k\right)}{N}}g\left(k-n\right)e{\frac{-2i\pi l\left(k-n\right)l}{N}}\vert \notag \right.\nonumber \\
&\left. +\vert \sum_{k=0}^{N-1} A_{k'} \left(k\right)\sin\left(\frac{2\pi \phi_{k'}g\left(k\right)}{N}\right)\left(k-n\right)e{\frac{-2i\pi l\left(k-n\right)l}{N}}\vert\right).\label{eq6}
\end{align}
By the fact that $A_{k'}$ is naturally a continous function and without a loss of generality we can assume $A_{k'}\left(-1\right)=0$. As a result, by the mean value theorem and by the inequality \eqref{eq4p} we can rewrite the second part of \eqref{eq6} as
\begin{shrinkeq}{-0.5ex}
\begin{align}
\vert & A_{k_0}\left(n\right)e^{\frac{2\pi i
\phi_{k_0}\left(n\right)}{N}}\hat{g}\left(l-\phi'_{k_0}\left(n\right)\right)-V_gx\left(n,l\right)\vert\nonumber\\
&\leq \sum_{k'=0}^{K} \epsilon \Vert{\phi'}_{k'}\Vert_{\infty} \left( I_1^1+\pi \vert A_{k'}\left(n\right) \vert I_2^1\right)\nonumber \\
&+\vert \sum_{k=0}^{N-1} A_{k'} \left(k\right)-A_{k'} \left(-1\right)\sin\left(\frac{2\pi \phi_{k'}\left(k\right)}{N}\right)g\left(k-n\right)\vert \nonumber\\
&\leq \sum_{k'=0}^{K} \left(\epsilon \Vert\phi '_{k'}\Vert_{\infty} \left( I_1^1+\pi \vert A_{k'}\left(n\right) \vert I_2^1\right) \notag \right.\nonumber\\
&\left.+\vert \sum_{k=0}^{N-1} \Vert{A'}_{k'}\Vert_{\infty} \vert k+1\vert +g\left(k-n\right)\vert\right)\nonumber\\
&\leq \sum_{k'=0}^{K} \left(\epsilon \Vert\phi '_{k'}\Vert_{\infty} \left( I_1^1+\pi \vert A_{k'}\left(n\right) \vert I_2^1\right) \notag \right.\nonumber\\
&\left.+\vert \sum_{k=0}^{N-1} \epsilon\Vert{\phi'}_{k'}\Vert_{\infty} \vert k+1\vert +g\left(k-n\right)\vert\right)\nonumber\\
&\leq \sum_{k'=0}^{K} \epsilon \Vert\phi '_{k'}\Vert_{\infty} \left( I_1^1+\pi \vert A_{k'}\left(n\right) \vert I_2^1+I_3^1\right).\label{eq7}
\end{align}
\end{shrinkeq}
Similarly to the procedure of \eqref{eq7} and based on the inequality \eqref{eq5} we can write the first part of \eqref{eq66} as
\begin{align}
\vert V_gx&\left(n+1,l\right)- e^{\frac{2\pi
i}{N}\left(\phi'_{k_0}\left(n\right)\right)}
A_{k_0}\left(n\right)e^{\frac{2\pi i
\phi_{k_0}\left(n\right)}{N}}\hat{g}\left(l-\phi'_{k_0}\left(n\right)\right)\vert.\nonumber\\
&\leq \sum_{k'=0}^{K} \epsilon \Vert\phi '_{k'}\Vert_{\infty} \left( I_1^2+\pi \vert A_{k'}\left(n\right) \vert I_2^2+I_3^2\right).\label{eq8}
\end{align}
By inequalities \eqref{eq66}, \eqref{eq7} and \eqref{eq8} we have 
\begin{align*}
&\vert\omega_x\left(n,l\right)-\phi'_{k_0}\left(n\right)\vert\\
& \leq \frac{NM\epsilon}{2\pi \delta}\left[ \sum_{k'=1}^K \Vert \phi'_{k'}\Vert_{\infty} \left( I_1^1+\pi \vert A_{k'}\left(n\right)\vert I_2^1+I_3^1 \notag \right.\right.\\
&\left.\left. + I_1^2+\pi \vert A_{k'}\left(n\right)\vert I_2^2+I_3^2\right)\right]+0.5=\tilde{\epsilon}.
\end{align*}
We can derive the inequalities \eqref{eqt3} and \eqref{eqt4} using the same
procedure as usin for the proof of Theorem \ref{th2}. 
\end{proof}
Now by the same argument that was used for the proof of Theorem \ref{th3} we
have:
\begin{theorem}\label{th4.2}
Consider a signal $x\in \mathcal{B}_d$ is given. Select a window function $g\in \mathbb{C}^N$ such that $\vert\mathrm{supp}\left(\hat{g}\right)\vert<\frac{d}{2}$. Moreover consider $\Vert\phi''_{k'}\Vert_{\infty}\leq \epsilon\Vert\phi'_{k'}\Vert_{\infty}$ and $\Vert A'_{k'}\Vert_{\infty}\leq \epsilon\Vert\phi'_{k'}\Vert_{\infty}$. Let $\delta=\min\{V_gx\left(n,l\right); 0\leq n,l \leq N-1, V_gx\left(n,l\right)\neq 0\}$ and $M=\max\{1,\frac{V_gx\left(n,l\right)}{V_gx\left(n+1,l\right)}\}$. Then we have
\[
A_{k'}\left(n\right) \cos\left(\frac{2\pi }{N} \phi_{k'}\left(n\right)\right)=\sum_{\vert\xi-\phi'_{k'}\left(n\right)\vert \leq \epsilon}S_gx\left(n,\xi\right),
\]
for all $n=0,1,\cdots,N-1$ and all $k'\in\{1,2,\cdots,K\}$.
\end{theorem}
\section{\bf{Consistency and Stability of Finite Synchrosqueezing Transform}}
In this section we show that the finite Synchrosqueezing transform has a
stability property similar to that shown for the continuous transform (shown in
\cite{Thakur}). 
To this purpose we present a theorem that shows the error of the instantaneous
frequency information obtained by formula \eqref{eq2} for noisy signal is less
than an $\tilde{\epsilon}$ which is given in the Theorem. 
\begin{theorem}
Consider a signal $x\in \mathcal{B}_d$ is given. Select a window function $g\in \mathbb{C}^N$ such that $\vert \mathrm{supp}\left(\hat{g}\right)\vert<\frac{d}{2}$. Moreover consider $\Vert\phi''_{k'}\Vert_{\infty}\leq \epsilon\Vert\phi'_{k'}\Vert_{\infty}$ and $\Vert A'_{k'}\Vert_{\infty}\leq \epsilon\Vert\phi'_{k'}\Vert_{\infty}$. Also, let $\delta=\min\{V_gx\left(n,l\right); 0\leq n,l \leq N-1, V_gx\left(n,l\right)\neq 0\}$ and $M=\max\{1,\frac{V_gx\left(n,l\right)}{V_gx\left(n+1,l\right)}\}$. 

Furthermore, suppose we have a noisy signal $e\in \mathbb{C}^N$ such that $\Vert e\Vert _{\infty}\Vert g \Vert_1\leq \epsilon'$. Letting $y=x+e$, the following statements holds for each $k'$ and $\left(n,l\right)\in Z_{k'}$

\begin{align*}
\vert &\omega_y\left(n,l\right)-\phi'_{k_0}\left(n\right)\vert\\
&\leq 
\frac{NM\epsilon}{2\pi \delta}\left[ \sum_{k'=1}^K \Vert \phi'_{k'}\Vert_{\infty} \left( I_1^1+\pi \vert A_{k'}\left(n\right)\vert I_2^1+I_3^1 \notag\right.\right.\\
&\left.\left. + I_1^2+\pi \vert A_{k'}\left(n\right)\vert I_2^2+I_3^2\right) \right]+0.5+\frac{N{M'}{\epsilon'}}{\pi} \\
&=\tilde{\epsilon}+\frac{NM'\epsilon'}{\pi},
\end{align*}

where $I_i^j$ for $i,j=1,2$ is denoted the same as in the section $3$ and $I_3^j$ for $j=1,2$ is denoted as
\[
I_3^j=\sum_{k=0}^{N-1} \vert g\left(k-n-j+1\right) \vert \vert k+1\vert.
\]

Furthermore, if $\left(n,l\right)\notin Z_{k'}$ for every $k'=1,\cdots,K$, then
\begin{align}\label{eqte3}
\vert V_gy\left(n,l\right)\vert\leq \sum_{k'=0}^{K} \epsilon \Vert\phi '_{k'}\Vert_{\infty} \left( I_1^1+\pi \vert A_{k'}\left(n\right) \vert I_2^1+I_3^1\right)+\frac{NM'\epsilon'}{\pi}.
\end{align}
and
\begin{align}\label{eqte4}
\vert V_gy\left(n+1,l\right)\vert\leq \sum_{k'=0}^{K} \epsilon \Vert\phi '_{k'}\Vert_{\infty} \left( I_1^2+\pi \vert A_{k'}\left(n\right) \vert I_2^2+I_3^2\right)+\frac{NM'\epsilon'}{\pi}.
\end{align}
\end{theorem}
\begin{proof}
By considering the assumptions, we have
\begin{align*}
&\vert \omega_y\left(n,l\right) -{\phi'}_{k'}\left(n\right)\vert\\
&\leq \vert \omega_y\left(n,l\right)-\tilde{\omega}_x\left(n,l\right)\vert+\vert \tilde{\omega}_x\left(n,l\right) -{\phi'}_{k'}\left(n\right)\vert,
\end{align*}
where $\tilde{\omega}_x\left(n,l\right)$ is the instantaneous frequency information of $x$ at the point $\left(n,l\right)$ without rounding. From Theorem \eqref{th4} we have

\begin{align}
& \vert\omega_x\left(n,l\right)-\phi'_{k_0}\left(n\right)\vert\nonumber\\
& \leq \frac{NM\epsilon}{2\pi \delta}\left[ \sum_{k'=1}^K \Vert \phi'_{k'}\Vert_{\infty} \left( I_1^1+\pi \vert A_{k'}\left(n\right)\vert I_2^1+I_3^1\notag \right.\right.\nonumber\\
&\left.\left. + I_1^2+\pi \vert A_{k'}\left(n\right)\vert I_2^2+I_3^2\right)
  \right]=\tilde{\epsilon}-0.5. \label{eqe1}
\end{align}

On the other hand
\begin{shrinkeq}{-0.4ex}
\begin{align}
&\vert \omega_y\left(n,l\right)-\omega_x\left(n,l\right)\vert\nonumber\\
&= \vert \mathrm{round}\left( \mathrm{real}\left(\frac{N}{2\pi i} \ln \frac{V_gy\left(n+1,l\right)}{V_gy\left(n,l\right)}\right)\right)\nonumber\\
&\quad-\mathrm{real} \left(\ln\frac{V_gx\left(n+1,l\right)}{V_gx\left(n,l\right)}\right)\vert\nonumber\\
&\leq \frac{N}{2\pi}\vert \ln \frac{V_gy\left(n+1,l\right)}{V_gy\left(n,l\right)}-\ln\frac{V_gx\left(n+1,l\right)}{V_gx\left(n,l\right)}\vert+0.5\nonumber\\
&\leq\frac{NM'}{2\pi}\vert V_gy\left(n+1,l\right)-V_gy\left(n,l\right)\nonumber\\
&\quad -V_gx\left(n+1,l\right)+V_gx\left(n,l\right)\vert+0.5\nonumber\\
&\leq\frac{NM'}{2\pi}\left(\vert V_gy\left(n+1,l\right)-V_gx\left(n+1,l\right)\vert\right.\nonumber\\
&\left. \quad +\vert V_gy\left(n,l\right)-V_gx\left(n,l\right)\vert \right)+0.5\nonumber\\
&\leq\frac{NM'}{2\pi}\left( \vert \sum_{k=0}^{N-1}e\left(k\right) g\left(k-n-1\right) e^{\frac{-2\pi ikl}{N}}\vert \notag \right.\nonumber\\
&\left. \quad +\vert \sum_{k=0}^{N-1}e\left(k\right) g\left(k-n\right) e^{\frac{-2\pi ikl}{N}}\vert\right)+0.5\nonumber\\
&\leq\frac{NM'}{2\pi} \left(2\Vert e\Vert _{\infty}\Vert g\Vert_1 \right)+0.5\nonumber\\
&\leq\frac{NM'\epsilon'}{\pi}+0.5.\label{eqe2}
\end{align}
\end{shrinkeq}
By the inequalities \eqref{eqe1} and \eqref{eqe2} we have
\[
\vert\omega_y\left(n,l\right)-\phi'_{k_0}\left(n\right)\vert \leq \frac{NM'\epsilon'}{\pi}+\tilde{\epsilon}.
\]
Furthermore, the inequalities \eqref{eqte3} and \eqref{eqte4} follow from
\eqref{eqe2} and \eqref{eqt3} and \eqref{eqt4}. 
\end{proof}
In the following theorem we will show how to reconstruct an oscillation from a noisy signal.
\begin{theorem} 
Consider a signal $x\in \mathcal{B}_d$ is given. Pick a window function $g\in \mathbb{C}^N$ such that $\vert\mathrm{supp}\left(\hat{g}\right)\vert<\frac{d}{2}$. Moreover consider $\Vert\phi''_{k'}\Vert_{\infty}\leq \epsilon\Vert\phi'_{k'}\Vert_{\infty}$ and $\Vert A'_{k'}\Vert_{\infty}\leq \epsilon\Vert\phi'_{k'}\Vert_{\infty}$. Also, let $\delta=\min\{V_gx\left(n,l\right); 0\leq n,l \leq N-1, V_gx\left(n,l\right)\neq 0\}$ and $M=\max\{1,\frac{V_gx\left(n,l\right)}{V_gx\left(n+1,l\right)}\}$.  

Furthermore, suppose we have a noisy signal $e\in \mathbb{C}^N$ such that $\Vert e\Vert _{\infty}\Vert g \Vert_1\leq \epsilon'$. For $y=x+e$, we have
\[
\vert A_{k'}\left(n\right) \cos\left(\frac{2\pi }{N} \phi_{k'}\left(n\right)\right)-\sum_{\vert\xi-\phi'_{k'}\left(n\right)\vert \leq \epsilon}S_gy\left(n,\xi\right)\vert\leq \epsilon',
\]
for all $n=0,1,\cdots,N-1$ and all $k'\in\{1,2,\cdots,K\}$.
\end{theorem}
\begin{proof}
From \eqref{eqp1} we have

\begin{align*}
&\vert A_{k'}\left(n\right) \cos\left(\frac{2\pi }{N} \phi_{k'}\left(n\right)\right)-\sum_{\vert\xi-\phi'_{k'}\left(n\right)\vert \leq \epsilon}S_gy\left(n,\xi\right)\vert\\
&\leq \vert A_{k'}\left(n\right) \cos\left(\frac{2\pi }{N} \phi_{k'}\left(n\right)\right)-\sum_{\{l;\vert\omega_x\left(n,l\right)-\phi'_{k'}\left(n\right)\vert \leq \tilde{\epsilon}\}}V_gy\left(n,l\right)\vert\\
&\leq \vert A_{k'}\left(n\right) \cos\left(\frac{2\pi }{N} \phi_{k'}\left(n\right)\right)\\
& -\sum_{\{l;\vert\omega_x\left(n,l\right)-\phi'_{k'}\left(n\right)\vert \leq \tilde{\epsilon}\}}V_gx\left(n,l\right)+V_ge\left(n,l\right)\vert.
\end{align*}

Now by the Theorem \ref{th4.2} we have

\begin{align*}
\vert A_{k'}\left(n\right) \cos&\left(\frac{2\pi 
\phi_{k'}\left(n\right)}{N}\right)\\
& -\sum_{\{l\vert\omega_x\left(n,l\right)-\phi'_{k'}\left(n\right)\vert \leq \tilde{\epsilon}\}}V_gx\left(n,l\right)+V_ge\left(n,l\right)\vert\\
&\leq \vert \sum_{\{l\vert\omega_x\left(n,l\right)-\phi'_{k'}\left(n\right)\vert \leq \tilde{\epsilon}\}}V_ge\left(n,l\right)\vert\\
&\leq \vert \sum_{k=0}^{N-1}e\left(k\right) g\left(k-n\right)e^{\frac{-2\pi ikl}{N}}\vert\\
&\leq \sum_{k=0}^{N-1}\vert e\left(k\right) g\left(k-n\right)e^{\frac{-2\pi ikl}{N}}\vert\\
&\leq \Vert e\Vert_{\infty} \sum_{k=0}^{N-1}\vert g\left(k-n\right)\vert\\
&\leq \epsilon' .
\end{align*}

\end{proof}
\section{\bf{Numerical Results}}

In this section, we apply the algorithms of sections \ref{s:STFTS} and
\ref{s:STFTSreal} for several test
cases. We consider a chirp signal, a multi-component signal and a signal with
interlacing instantaneous frequency elements and these signals with noise. We
compute the instantaneous frequency of these signals using the finite STFT
Synchrosqueezing transform. For the instantaneous-frequency computation, we use
a window function such that its Fourier transform is a Hann function with a 
support of $10$ samples.  

To evaluate the performance of finite STFT Synchrosqueezing transform, we
compare the time varying power spectrum of the finite STFT Synchrosqueezing
transform and the finite STFT transform with the ideal time-varying power
spectrum(itvPS) of the test signals. We define the itvPS of the signal of the form
\eqref{eq1} as: 
\begin{equation}\label{itvPS}
P_x\left(n,\xi\right)=\sum_{k'=1}^K
A_{k'}\left(n\right)^2\delta\left(\xi-\phi'_{k'}\left(n\right)\right).
\end{equation}

%
 
 The first signal is a single-component chirp signal
 \begin{equation}\label{sig1}
 x\left(n\right)=\cos\left(2\pi \left( \frac{n}{20}+0.05\left(\frac{n}{20}\right)^2\right)\right),
 \end{equation}
 for $n=0,\cdots,199$. The phase function of \eqref{sig1} is equal to 
\[
\phi\left(n\right)=10n+\frac{1}{40}n^2,
\]
and the instantaneous frequency is 
\begin{equation}\label{iff0}
{\phi}'\left(n\right)=10+\frac{1}{20}n.
\end{equation}
 The finite STFT and the finite STFT Synchrosqueezing transform of \eqref{sig1} are shown in Figure \ref{fig1}.

\begin{figure}[ht] 
\centering 
\includegraphics[height=0.7\textheight]{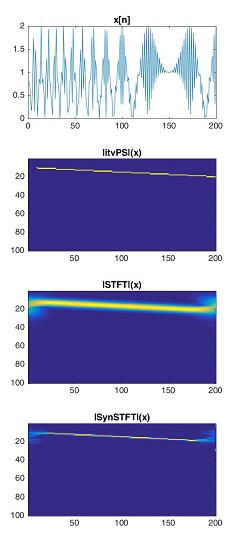} 
\caption{Top: The chirp signal \eqref{sig1} First middle: The itvPS of the signal $x$ Second middle: The finite STFT of the signal$x$. Bottom: The finite STFT Synchrosqueezing transform of the signal $x$. }\label{fig1} 
\end{figure} 
It is seen from Figure \ref{fig1} that the finite STFT Synchrosqueezing 
transform has a better estimation of the instantaneous frequency given in \eqref{iff0}.

The second signal is a two-component signal that is given by
\begin{equation}\label{sig2}
\begin{aligned}
&x\left(n\right)=\cos\left(2\pi \left( \frac{n}{10}+0.2\left(\frac{n}{10}\cos \left(\frac{n}{10}\right)\right)\right)\right)\\
&+\cos\left(2\pi \left( \frac{3n}{10}+0.02\left(\frac{n}{10}\right)^2\right)\right),
\end{aligned}
\end{equation}
for $n=0,\cdots,199$. The phase functions of the signal \eqref{sig2} are equal to
\[
\phi_1\left(n\right)=40n+4\cos\left(\frac{n}{10}\right),
\]
and
\[
\phi_2\left(n\right)=60n+\frac{4}{100}n^2.
\]
Consequently, the instantaneous frequencies are equal to
\begin{equation}\label{iff1}
\phi'_1\left(n\right)=40-\frac{4}{10}\sin\left(\frac{n}{10}\right),
\end{equation}
and
\begin{equation}\label{iff2}
\phi'_2\left(n\right)=60+\frac{8}{100}n.
\end{equation}
The finite STFT and the finite STFT Synchrosqueezing
transform of \eqref{sig2} are shown in Figure \ref{fig2}. Note that the support
of the window function in the frequency domain is $10$ samples which is less 
than the separation of the two consequative instantaneous frequencies, which is
20 samples.  
\begin{figure}[ht] 
\centering 
\includegraphics[ height=0.7\textheight]{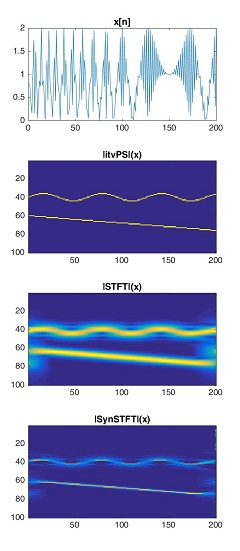} 
\caption{Top: The two component signal \eqref{sig2} . First middle: The itvPS of the signal $x$. Second middle: The finite STFT of the signal$x$. Bottom: The finite STFT Synchrosqueezing transform of the signal $x$.}\label{fig2} 
\end{figure} 

Figure \ref{fig2} shows that the energy in the finite STFT Synchrosqueezing
transform is better concentrated around the instantaneous frequencies
\eqref{iff1} and \eqref{iff2} as compared to the finite STFT. 

The third signal is a signal with interlacing frequency elements
\begin{equation}\label{sig3}
x\left(n\right)=\cos\left( 5\pi \left( \frac{n}{10}\right)\right)+\cos\left(2\pi \left( \frac{n}{10}+0.05\left(\frac{n}{10}\right)^2\right)\right),
\end{equation}
for $n=0,\cdots,199$. The phase functions of the signal \eqref{sig3} are equal to
\[ 
\phi_1\left(n\right)=50n,
\]
and
\[
\phi_2\left(n\right)=20n+\frac{1}{10}n^2,
\] 
which results in the instantaneous frequencies
\[
\phi'_1\left(n\right)=50,
\]
and
\[
\phi'_2\left(n\right)=20+\frac{1}{5}n.
\]
The finite STFT and the finite STFT Synchrosqueezing transform of \eqref{sig3}
are shown in Figure \ref{fig3}. Although the signal in Figure \ref{fig3} is not in $\mathcal{B}_d$, the result is well separated.

\begin{figure}[ht] 
\centering 
\includegraphics[ height=0.7\textheight]{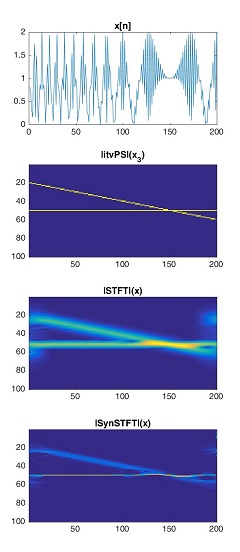} 
\caption{Top: The signal with interlacing frequency elements \eqref{sig3} First middle: The itvPS of the signal $x$. Second middle: The finite STFT of the signal$x$. Bottom: Finite STFT Synchrosqueezing transform of the signal $x$.}\label{fig3} 
\end{figure} 

We also added some noise to these three signals before applying the finite STFT
Synchrosqueezing transform and finite STFT transform. The results are shown in
figures \ref{fig1e}, \ref{fig2e} and \ref{fig3e}. It is seen that the finite
STFT Synchrosqueezing transform is more robust to noise than finite STFT
transform.  
\begin{figure}[ht] 
\centering 
\includegraphics[ height=0.7\textheight]{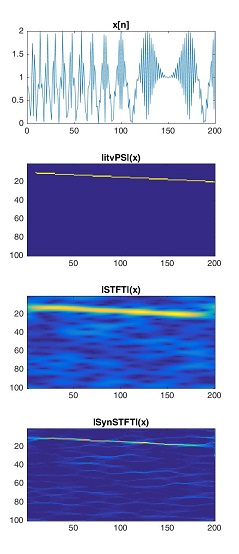} 
\caption{Top: The noisy signal $y=x+e$ where $\Vert e \Vert_{\infty}=0.4$ and $x$ as defined in \eqref{sig1}. First middle: The itvPS of the signal $y$. Second middle: The finite STFT of the noisy signal $y$. Bottom: Finite STFT Synchrosqueezing transform of the noisy signal $y$.}\label{fig1e} 
\end{figure} 
\begin{figure}[ht] 
\centering 
\includegraphics[ height=0.7\textheight]{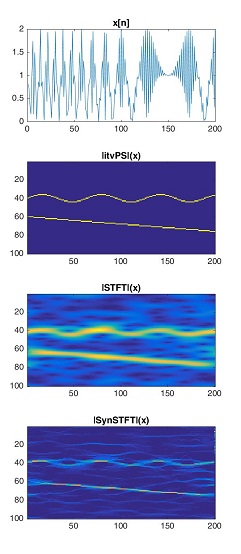} 
\caption{Top: The noisy signal $y=x+e$ where $\Vert e \Vert_{\infty}=0.4$ and $x$ as defined in \eqref{sig2}. First middle: The itvPS of the signal $y$. Second middle: The finite STFT of the noisy signal $y$. Botton: The finite STFT Synchrosqueezing transform of the noisy signal $y$.}\label{fig2e} 
\end{figure} 
\begin{figure}[ht] 
\centering 
\includegraphics[ height=0.7\textheight]{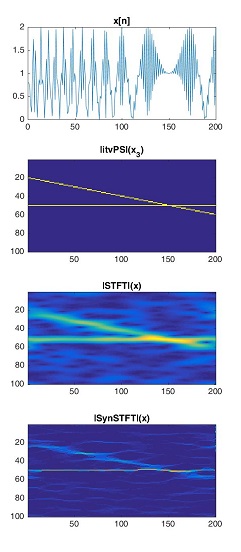} 
\caption{Top: The noisy signal $y=x+e$ where $\Vert e \Vert_{\infty}=0.4$ and $x$ as defined in \eqref{sig3}. First middle: The itvPS of the signal $y$. Second middle: The finite STFT of the noisy signal $y$. Botton: The finite STFT Synchrosqueezing transform of the noisy signal $y$.}\label{fig3e} 
\end{figure}

 \bibliographystyle{IEEEtran}

\end{document}